\def\coralreport{0}

\documentclass[10pt]{article}

\usepackage{dpr_fec,ifthen,dsfont,enumitem}
\allowdisplaybreaks

\newcommand{\ones}{\mathds{1}}
\newcommand{\epsopt}{\epsilon_\mathrm{opt}}
\newcommand{\nuopt}{\nu_\mathrm{opt}}
\newcommand{\rednu}{\theta_\nu}
\newcommand{\redeps}{\theta_\epsilon}
\newcommand{\clarkesub}{{\bar\partial}}
\newcommand{\clarkeepssub}{\clarkesub_\epsilon}
\newcommand{\hifoo}{{\sc hifoo}}
\newcommand{\hanso}{{\sc hanso}}
\newcommand{\Hifoo}{{\sc Hifoo}}
\newcommand{\red}[1]{\textcolor{red}{#1}}

\ifthenelse{\coralreport = 1}{
  \usepackage{isetechreport}

  \coraltrue
  \cvcrfalse
}{\usepackage{fullpage}}

\usepackage{amsmath,amsthm,amssymb,amsfonts}

\usepackage{graphicx}

\usepackage{booktabs}

\usepackage{rotating}

\usepackage{colortbl}

\usepackage{textcomp}

\usepackage{subfigure}

\usepackage{url}

\usepackage{hyperref}

\usepackage{framed}

\usepackage{tikz}
\usepackage{pgfplots}

\usepackage{listings}
\usepackage[numbers,sort,compress]{natbib}

\usepackage{color}
\usepackage{xcolor}

\usepackage{enumitem}

\DeclareMathOperator*{\co}{conv}

\usepackage[margin=10pt,font=small,labelfont=bf, labelsep=endash]{caption}

\newtheorem{prop}{Property}
\newtheorem{examp}{Example}

\numberwithin{corollary}{section}

\begin{document}

\title{Gradient Sampling Methods for Nonsmooth Optimization}

  \author{J.V.~Burke\thanks{Department of Mathematics, University of Washington, Seattle, WA.\ \texttt{jvburke01@gmail.com}.\ Supported in part by the U.S.\ National Science Foundation grant DMS-1514559.} 
  \and F.E.~Curtis\thanks{Department of Industrial and Systems Engineering, Lehigh University, Bethlehem, PA.\ \texttt{frank.e.curtis@gmail.com}.\ Supported in part by the U.S.\ Department of Energy grant DE-SC0010615.}
  \and A.S.~Lewis\thanks{School of Operations Research and Information Engineering, Cornell University, Ithaca, NY.\ \texttt{adrian.lewis@cornell.edu}.\ Supported in part by the U.S.\ National Science Foundation grant DMS-1613996. }
  \and M.L.~Overton\thanks{Courant Institute of Mathematical Sciences, New York University.\ \texttt{mo1@nyu.edu}.\ Supported in part by the U.S.\ National Science Foundation grant DMS-1620083.}
  \and L.E.A.~Sim\~oes\thanks{Department of Applied Mathematics, University of Campinas, Brazil.\ \texttt{simoes.lea@gmail.com}.\ Supported in part by the S\~ao Paulo Research Foundation (FAPESP), Brazil, under
  grants 2016/22989-2 and 2017/07265-0.}
  }

\maketitle

\centerline{Submitted to: \emph{Special Methods for Nonsmooth Optimization}, Springer, 2018}
\centerline{A.~Bagirov, M.~Gaudioso, N.~Karmitsa and M.~M\"akel\"a, eds.}

\bigskip

\begin{abstract}
This paper reviews the gradient sampling methodology for solving nonsmooth, nonconvex optimization problems.  An intuitively straightforward gradient sampling algorithm is stated and its convergence properties are summarized.  Throughout this discussion, we emphasize the simplicity of gradient sampling as an extension of the steepest descent method for minimizing smooth objectives.  We then provide overviews of various enhancements that have been proposed to improve practical performance, as well as of several extensions that have been made in the literature, such as to solve constrained problems.  The paper also includes clarification of certain technical aspects of the analysis of gradient sampling algorithms, most notably related to the assumptions one needs to make about the set of points at which the objective is continuously differentiable. Finally, we discuss possible future research directions.
\end{abstract}

\section{Introduction}\label{sec.introduction}

The Gradient Sampling (GS) algorithm is a conceptually simple descent method for solving nonsmooth, nonconvex optimization problems, yet it is one that possesses a solid theoretical foundation and has been employed to substantial success in a wide variety of applications.  Since the appearance of the fundamental algorithm and its analysis little over a dozen years ago, GS has matured into a comprehensive methodology.  Various enhancements have been proposed that make it a competitive approach in many nonsmooth optimization contexts, and it has been extended in various interesting ways, such as for nonsmooth optimization on manifolds and for constrained problems.  The purpose of this work is to provide background and motivation for the development of the GS method, discuss its theoretical guarantees, and provide an overview of the enhancements and extensions that have been the subject of research over recent years.

The underlying philosophy of GS is that virtually any nonsmooth objective function of interest is differentiable almost everywhere; in particular, this is true if the objective $f : \R{n} \to \R{}$ is either locally Lipschitz continuous or semialgebraic.  In such cases, when $f$ is evaluated at a randomly generated point $x \in \R{n}$, it is differentiable there with probability one.  This means that an algorithm can rely on an ability to obtain the objective function value $f(x)$ and gradient~$\nabla f(x)$, as when $f$ is smooth, rather than require an oracle to compute a subgradient.  In most interesting settings, $f$ \emph{is not} differentiable at its local minimizers, but, under reasonable assumptions, the carefully crafted mechanisms of the GS algorithm generate a sequence of iterates---at which~$f$ \emph{is} differentiable---converging to stationarity.

At the heart of GS is a stabilized steepest descent approach.  When $f$ is differentiable at~$x$, the negative gradient $-\nabla f(x)$ is, of course, the traditional steepest descent direction for~$f$ at~$x$ in the 2-norm in that
\bequation\label{eq.steepest_descent}
  -\frac{\nabla f(x)}{\|\nabla f(x)\|_2} = \arg\min_{\|d\|_2 \leq 1}\ \nabla f(x)^Td.
\eequation
However, when $x$ is near a point where $f$ is not differentiable, it may be necessary to take a very short step along $-\nabla f(x)$ to obtain decrease in $f$.  It is for this reason that the traditional steepest descent method may converge to nonstationary points when $f$ is nonsmooth.\footnote{Although this fact has been known for decades, most of the examples that appear in the literature are rather artificial since they were designed with exact line searches in mind.  Analyses showing that the steepest descent method with inexact line searches converges to nonstationary points of some simple convex nonsmooth functions have appeared recently in~\cite{AslOver17,GuoLewi17}.}  The GS algorithm stabilizes the choice of the search direction to avoid this issue.  In each iteration, a descent direction from the current iterate $x^k$ is obtained by supplementing the information provided by $\nabla f(x^k)$ with gradients evaluated at randomly generated points $\{x^{k,1},\dots,x^{k,m}\} \subset \Bmbb(x^k,\epsilon_k) := \{x \in \R{n} : \|x - x^k\|_2 \leq \epsilon_k\}$, which are \emph{near} $x^k$, and then computing the minimum-norm vector $g^k$ in the convex hull of these gradients.  This choice can be motivated by the goal that, with $\clarkeepssub f(x)$ 
denoting the Clarke $\epsilon$-subdifferential of $f$ at $x$ (see \S\ref{sec.theory}),
\bequation
  -\frac{g^k}{\|g^k\|_2} \approx \arg\min_{\|d\|_2 \leq 1} \max_{g\in\clarkesub f_\epsilon(x)} g^Td;
\eequation
i.e., $-g^k$ can essentially be viewed as a steepest descent direction for $f$ from $x^k$ in a more ``robust'' sense.  A line search is then used to find a positive stepsize $t_k$ yielding decrease in $f$, i.e., $f(x^k - t_kg^k) < f(x^k)$.  The sampling radius $\epsilon_k$ that determines the meaning of ``\emph{near} $x^k$'' may either be fixed or adjusted dynamically.

A specific instance of the GS algorithm is presented in \S\ref{sec.algorithm}.  Its convergence guarantees are summarized in~\S\ref{sec.theory}.  We then present various enhancements and extensions of the approach in \S\ref{sec.enhancements} and \S\ref{sec.extensions}, respectively, followed by a discussion of some successful applications of the GS methodology in \S\ref{sec.applications}.  Throughout this work, our goal is to emphasize the \emph{simplicity} of the fundamental GS strategy.  We believe that this, in addition to its strong convergence properties for locally Lipschitz optimization, makes it an attractive choice when attempting to solve difficult types of nonsmooth optimization problems.

Although the first convergence analysis of a GS algorithm was given by Burke, Lewis, and Overton in \cite{BurkLewiOver05}, an earlier version of the method was presented by these authors in~\cite{BurkLewiOver02b}. That algorithm, originally called a ``gradient bundle'' method, was applied to a function that was not only nonconvex and nonsmooth, but also non-locally-Lipschitz, namely, the spectral abscissa---i.e., the largest of the real parts of the eigenvalues---of a linear matrix function $A$ mapping a parameter vector $x$ to the space of nonsymmetric square matrices.  The spectral abscissa is not locally Lipschitz at a matrix $\bar X$ when an eigenvalue of $\bar X$ with largest real part has multiplicity two or more \cite{BurkOver01}, but it is semialgebraic and, hence, differentiable almost everywhere, so a GS algorithm was applicable.  The method was surprisingly effective.  As anticipated, in most cases the apparent local minimizers that were approximated had the property that the eigenvalues of $A$ with largest real part had multiplicity two or more.  An illustration that appeared in \cite{BurkLewiOver02b} is reproduced in Figure~\ref{grad_bundle}; the extremely ``steep" contours of the objective function indicate its non-Lipschitzness.  Obtaining theoretical results for a GS algorithm applied to non-locally-Lipschitz problems seems very
challenging; we discuss this issue further in \S\ref{sec.nonlip}, after describing the substantial body of theory that has
been developed for the locally Lipschitz case in \S\ref{sec.global} and \S\ref{sec.local}.

\begin{figure}[ht]
  \centering
  \includegraphics[height=2.5in]{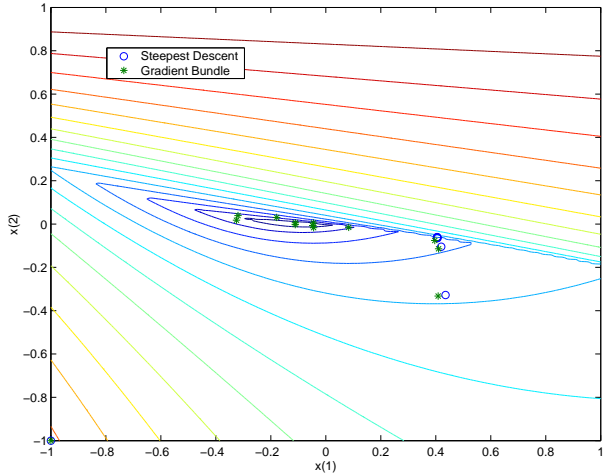}
  \caption{Contours of the spectral abscissa of an affine matrix family given in \cite{BurkLewiOver02b}.  Iterates of the ordinary gradient (``steepest descent") method with a line search are shown (small circles) along with those of the gradient sampling (``gradient bundle") algorithm (asterisks).  Both start at $(-1,-1)$.}
  \label{grad_bundle}
\end{figure}

\section{Algorithm GS}\label{sec.algorithm}

We now state a specific variant of the GS algorithm. We start by assuming only that the objective function~$f$ is locally Lipschitz over $\R{n}$, which implies, by Rademacher's theorem \cite{Clar83}, that $f$ is differentiable almost everywhere.  As previously mentioned, at the heart of the algorithm is the computation of a descent direction by finding the minimum norm element of the convex hull of gradients obtained about each iterate.  The remaining procedures relate to the line search to obtain decrease in $f$ and the selection of a subsequent iterate so that $f$ is differentiable at $\{x^k\}$.

\balgorithm[ht]
  \renewcommand{\thealgorithm}{GS}
  \caption{(Gradient Sampling)}
  \label{alg.gs}
  \balgorithmic[1]
    \Require initial point $x^0$ at which $f$ is differentiable, initial sampling radius $\epsilon_0 \in (0,\infty)$, initial stationarity target $\nu_0 \in [0,\infty)$, sample size $m \geq n+1$, line search parameters $(\beta,\gamma) \in (0,1) \times (0,1)$, termination tolerances $(\epsopt,\nuopt) \in [0,\infty) \times [0,\infty)$, and reduction factors $(\redeps,\rednu) \in (0,1] \times (0,1]$
    \For{$k \in \N{}$}
      \State independently sample $\{x^{k,1},\dots,x^{k,m}\}$ uniformly from $\Bmbb(x^k,\epsilon_k)$ \label{step.sample}
      \State compute $g^k$ as the solution of $\min_{g\in\Gcal^k} \thalf \|g\|_2^2$, where $\Gcal^k := \conv\{\nabla f(x^k),\nabla f(x^{k,1}),\dots,\nabla f(x^{k,m})\}$ \label{step.qp}
      \State \textbf{if} $\|g^k\|_2 \leq \nuopt$ and $\epsilon_k \leq \epsopt$ \textbf{then} terminate \label{step.terminate}
      \State \textbf{if} $\|g^k\|_2 \leq \nu_k$
      \State \qquad \textbf{then} set $\nu_{k+1} \gets \rednu\nu_k$, $\epsilon_{k+1} \gets \redeps\epsilon_k$, and $t_k \gets 0$
      \State \qquad \textbf{else}  set $\nu_{k+1} \gets \nu_k$, $\epsilon_{k+1} \gets \epsilon_k$, and \label{step.armijo}
      \bequation\label{eq.armijo}
        t_k \gets \max\left\{ t \in \{1,\gamma,\gamma^2,\dots\} : f(x^k - t g^k) < f(x^k) - \beta t \|g^k\|_2^2 \right\}
      \eequation
      \State \textbf{if} $f$ is differentiable at $x^k - t_k g^k $ \label{step.check}
      \State \qquad \textbf{then} set $x^{k+1} \gets x^k - t_k g^k$
      \State \qquad \textbf{else} set $x^{k+1}$ randomly as any point where $f$ is differentiable such that \label{step.update}
      \bequation\label{eq.update}
        f(x^{k+1}) < f(x^k) - \beta t_k \|g^k\|_2^2\ \ \text{and}\ \ \|x^k - t_k g^k - x^{k+1}\|_2 \leq \min\{t_k,\epsilon_k\} \|g^k\|_2
      \eequation
    \EndFor
  \ealgorithmic
\ealgorithm

While the essence of the methods from \cite{BurkLewiOver02b} and \cite{BurkLewiOver05} remains intact, Algorithm~\ref{alg.gs} differs in subtle yet important ways from the methods presented in these papers, as we now explain.
\begin{enumerate}
  \item Algorithm~\ref{alg.gs} incorporates a key modification proposed by Kiwiel in \cite[Alg.~2.1]{Kiwi07}, namely, the second inequality in \eqref{eq.update}; the version in \cite{BurkLewiOver05} used $\epsilon_k$ instead of $\min\{t_k,\epsilon_k\}$.  As Kiwiel explained, this minor change allowed him to drop the assumption in \cite{BurkLewiOver05} that the level set $\{ x: f(x) \leq f(x^0)\}$ is compact, strengthening the convergence results for the algorithm.
  \item A second change suggested in \cite[Alg.~2.1]{Kiwi07} is the introduction of the termination tolerances $\nuopt$ and $\epsopt$.  These were used in the computational experiments in \cite{BurkLewiOver05}, but not in the algorithm statement or analysis.  Note that if $\epsopt$ is set to zero, then Algorithm~\ref{alg.gs} never terminates since $\epsilon_k$ can never be zero, though it may happen that one obtains $\|g^k\|_2 = 0$.
  \item A third change, also suggested by Kiwiel, is the usage of the nonnormalized search direction $-g^k$ (originally used in \cite{BurkLewiOver02b}) instead of the normalized search direction $-g^k/\|g^k\|_2$ (used in \cite{BurkLewiOver05}). The resulting inequalities in \eqref{eq.armijo} and \eqref{eq.update} are taken from \cite[Sec.~4.1]{Kiwi07}. This choice does not affect the main conclusions of the convergence theory as
in both cases it is established \cite{BurkLewiOver05,Kiwi07} that the stepsize $t_k$ can be determined by a finite process. 
However, since Theorem \ref{th.eps_to_zero} below shows that a subsequence of $\{g^k\}$ converges to zero under reasonable conditions, one expects that fewer function evaluations should be required by the line search asymptotically when using the nonnormalized search direction,
whereas using the normalized direction may result
in the number of function evaluations growing arbitrarily large \cite[Sec.~4.1]{Kiwi07}. Furthermore, our practical experience is 
consistent with this viewpoint. 
  
  \item Another aspect of Algorithm~\ref{alg.gs} that is different in both \cite{BurkLewiOver05} and \cite{Kiwi07} concerns the randomization procedure in Step~\ref{step.sample}.  In the variants given in those papers, it was stated that the algorithm terminates if $f$ is not \emph{continuously} differentiable at the randomly sampled points $\{x^{k,1},\dots,x^{k,m}\}$.  In the theorem stated in the next section, we require only that $f$ is differentiable at the sampled points.  Since by Rademacher's theorem and countable additivity of probabilities
this holds for every sampled point with probability one, we do not include a termination condition here. 
  \item Finally, Steps \ref{step.check}--\ref{step.update} in Algorithm~\ref{alg.gs} \emph{do} require explicit checks that ensure that $f$ is differentiable at~$x^{k+1}$, but unlike in the variants in \cite{BurkLewiOver05} and \cite{Kiwi07}, it is not required that $f$ be \emph{continuously} differentiable at $x^{k+1}$. This differentiability requirement is included since it is not the case that $f$ is differentiable at $x^k - t_k g^k$ with probability one, as is shown via an example in \cite{HeloSantSimo16}, discussed further in \S\ref{subsec.avoidcheck}. For a precise procedure for implementing Step~\ref{step.update}, see \cite{Kiwi07}.
\end{enumerate}

The computation in Step~\ref{step.qp} requires solving a strongly convex quadratic optimization problem (QP) to compute the minimum-norm element of the convex hull of the current and sampled gradients, or, equivalently, to compute the 2-norm projection of the origin onto this convex hull.  It is essentially the same operation required in every iteration of a bundle method. To see this, observe that solving the QP in Step~\ref{step.qp} can be expressed, with 
\bequationNN
  G_k := \bbmatrix \nabla f(x^k) & \nabla f(x^{k,1}) & \cdots & \nabla f(x^{k,m}) \ebmatrix,
\eequationNN
as computing $(z_k,d^k,\lambda^k) \in \R{} \times \R{n} \times \R{m+1}$ as the primal-dual solution of the QPs
\bequation\label{prob.qp}
  \left\{
    \baligned
      \min_{(z,d) \in \R{} \times \R{n}} &\ z + \thalf \|d\|_2^2 \\
      \st &\ G_k^Td \leq z \ones
    \ealigned
  \right\}
  \quad
  \text{and}
  \quad
  \left\{
    \baligned
      \min_{\lambda \in \R{m+1}} &\ \thalf \|G_k\lambda\|_2^2 \\
      \st &\ \ones^T\lambda = 1,\ \lambda \geq 0
    \ealigned
  \right\}.
\eequation
The latter problem, yielding $G_k\lambda^k = g^k$, can easily be seen to be equivalent to solving the subproblem $\min_{g\in\Gcal^k} \thalf \|g\|_2^2$ stated in Step~\ref{step.qp}, whereas the former problem, yielding $d^k = -g^k$, can be seen to have the same form as the subproblems arising in bundle methods.

Normally, the initial stationarity target $\nu_0$ is chosen to be positive and the reduction factors $\rednu$ and $\redeps$ are chosen to be less than one so that the stationarity target and sampling radius are reduced every time the condition $\|g^k\|\leq\nu_k$ is satisfied. However, it is also interesting to consider the variant with $\nu_0=0$ and $\redeps=1$, forcing the algorithm to run forever with $\epsilon$ fixed unless it terminates with $g^k=0$ for some $k \in \N{}$.  We consider both of these variants in the global convergence theory given in the next section.

\section{Convergence Theory for Algorithm~\ref{alg.gs}}\label{sec.theory}

We begin with some definitions.  In the locally Lipschitz case, the Clarke subdifferential of $f$ at $x \in \R{n}$ is 
defined by the convex hull of the limits of gradients of $f$ on sequences converging to $x$ \cite[Def.\ 1.1]{Clar75}, i.e.,
\bequation\label{eq.clarkesub}
  \clarkesub f(x) = \conv \left\{ \lim_{j\to\infty} \nabla f(y^j) : \{y^j\} \to x\ \text{where } f \text{ is differentiable at } y^j \text{ for all }  j \in \N{} \right\}.
\eequation
A point $x$ is Clarke stationary for $f$ if $0 \in \clarkesub f(x)$.  A more ``robust'' sense of stationarity with respect to $f$ can be defined by considering the limits of gradients corresponding to limiting points \emph{near} $x$; in particular, given a radius $\epsilon \in [0,\infty)$, the Clarke $\epsilon$-subdifferential \cite{Gold77} is given by\footnote{The definition in \cite{BurkLewiOver05} includes a closure operation but this is unnecessary.}
\bequation\label{eq.clarkeepssub}
  \clarkeepssub f(x) = \conv \{\clarkesub f(\Bmbb(x,\epsilon))\}.
\eequation
A point $x$ is Clarke $\epsilon$-stationary for $f$ if $0 \in \clarkeepssub f(x)$.  For all practical purposes, one cannot generally evaluate $\clarkesub f$ (or $\clarkeepssub f$) at (or near) any point where $f$ is not differentiable.  That said, Algorithm~\ref{alg.gs} is based on the idea that one can approximate the minimum norm element in $\clarkeepssub f(x^k)$ through random sampling of gradients in the ball $\Bmbb(x^k,\epsilon)$. To a large extent this idea is motivated by \cite{BurkLewiOver02a} which investigates how well the entire Clarke subdifferential $\clarkesub f(x)$
can be approximated through random sampling. However, the results in \cite{BurkLewiOver02a} cannot be directly exploited in the analysis of
the GS algorithm because the gradients are sampled only at a finite number of points near any given iterate.

\subsection{Global Convergence Guarantees}\label{sec.global}

A critical aspect of theoretical convergence guarantees for Algorithm~\ref{alg.gs} concerns the set of points where~$f$ is \emph{continuously} differentiable, which we denote by $D$.  Consideration of $D$ played a crucial role in the analysis in both \cite{BurkLewiOver05} and \cite{Kiwi07}, but there were some oversights concerning both the requirements of the algorithm with respect to $D$ and the assumptions on $D$. Regarding the requirements of the algorithm with respect to $D$, there is actually no need, from a theoretical point of view, for either the iterates $\{x^k\}$ or the randomly generated sampled points $\{x^{k,j}\}$ to lie in $D$; all that is needed is that $f$ is differentiable at these points.  Most implementations of GS algorithms do not attempt to check any form of differentiability in any case, but if one were to attempt to implement such a check, it is certainly more tractable to check for differentiability than continuous differentiability.  Regarding the assumptions on $D$, in the theorems that we state below, we assume that $D$ is an open set with full measure in $\R{n}$. In contrast, the relevant assumption stated in \cite{BurkLewiOver05,Kiwi07} is weaker, namely, that $D$ is an open dense subset of $\R{n}$. However, the proofs of convergence actually require the full measure assumption on $D$ that we include below.\footnote{This oversight went unnoticed for 12 years until J.~Portegies and T.~Mitchell brought it to our attention recently.}

There are three types of global convergence guarantees of interest for Algorithm~\ref{alg.gs}: one when the input parameters ensure that $\{\epsilon_k\} \searrow 0$, one when $\epsilon_k $ is repeatedly reduced but a positive stopping tolerance prevents it from converging to zero, and one when $\epsilon_k=\epsilon > 0$ for all $k$.  These lead to  different properties for the iterate sequence.  The first theorem below relates to cases when the stationarity tolerance and sampling radius tolerance are both set to zero so that the algorithm can never terminate.

\btheorem\label{th.eps_to_zero}
  Suppose that $f$ is locally Lipschitz in $\R{n}$ and continuously differentiable on an open set $D$ with full measure in $\R{n}$.  Suppose further that Algorithm~\ref{alg.gs} is run with $\nu_0 > 0$, $\nuopt = \epsopt = 0$, and strict reduction factors $\rednu < 1$ and $\redeps < 1$.  Then, with probability one, Algorithm~\ref{alg.gs} is well defined in the sense that the gradients in Step \ref{step.qp} exist in every iteration, the algorithm does not terminate, and either
  \benumerate
    \item[(i)] $f(x^k) \searrow -\infty$, or
    \item[(ii)] $\nu_k \searrow 0$, $\epsilon_k\searrow 0$, and every cluster point of $\{x^k\}$ is Clarke stationary for $f$.
  \eenumerate
\etheorem

Theorem~\ref{th.eps_to_zero} is essentially the same as \cite[Thm 3.3]{Kiwi07} (with the modifications given in~\cite[\S4.1]{Kiwi07} for nonnormalized directions), except for two aspects:
\benumerate
  \item The proof given in \cite{Kiwi07} implicitly assumes that $D$ is an open set with full measure, as does the proof of \cite[Thm 3.4]{BurkLewiOver05} on which Kiwiel's proof is based, although the relevant assumption on $D$ in both papers is the weaker condition that $D$ is an open dense set. Details are given in Appendix \ref{app.fullmeasure}.

\item In the algorithms analyzed in \cite{Kiwi07} and \cite{BurkLewiOver05}, the iterates $\{x^k\}$ and the randomly sampled points $\{x^{k,j}\}$ were enforced to lie in the set $D$ where $f$ is continuously differentiable.  We show in Appendix \ref{app.fdifOK} that the theorem still holds when this requirement is relaxed to ensure only that~$f$ is differentiable at these points.

\eenumerate

As Kiwiel argues, Theorem~\ref{th.eps_to_zero} is essentially the best result that could be expected.  Furthermore, as pointed out in \cite[Remark 3.7(ii)]{Kiwi07},
it leads immediately to the following corollary.

\bcorollary
  Suppose that $f$ is locally Lipschitz in $\R{n}$ and continuously differentiable on an open set $D$ with full measure in $\R{n}$.  Suppose further that Algorithm~\ref{alg.gs} is run with $\nu_0 > \nuopt > 0$, $\epsilon_0 > \epsopt > 0$, and strict reduction factors $\rednu < 1$ and $\redeps < 1$.  Then, with probability one, Algorithm~\ref{alg.gs} is well defined in the sense that the gradients in Step \ref{step.qp} exist at every iteration, and either
  \benumerate
    \item[(i)] $f(x^k) \searrow -\infty$, or
    \item[(ii)] Algorithm~\ref{alg.gs} terminates by the stopping criteria in Step~\ref{step.terminate}.
  \eenumerate
\ecorollary

The final result that we state concerns the case when the sampling radius is fixed.  A proof of this result is essentially given by that of \cite[Thm 3.5]{Kiwi07}, again taking into account the comments in Appendices \ref{app.fullmeasure}
and \ref{app.fdifOK}.

\btheorem\label{th.eps_pos} 
  Suppose that $f$ is locally Lipschitz in $\R{n}$ and continuously differentiable on an open set $D$ with full measure in $\R{n}$.  Suppose further that Algorithm~\ref{alg.gs} is run with $\nu_0 = \nuopt = 0$, $\epsilon_0 = \epsopt = \epsilon > 0$, and $\redeps = 1$.  Then, with probability one, Algorithm~\ref{alg.gs} is well defined in the sense that the gradients in Step~\ref{step.qp} exist at every iteration, and one of the following occurs:
  \benumerate
    \item[(a)] $f(x^k) \searrow -\infty$, or
    \item[(b)]  Algorithm~\ref{alg.gs} terminates for some $k \in \N{}$ with $g^k = 0$, or
    \item[(c)] there exists $\Kcal \subseteq \N{}$ with $\{g^k\}_{k\in \Kcal} \to 0$ and every cluster point of $\{x^k\}_{k\in \Kcal}$ is Clarke $\epsilon$-stationary for $f$.
  \eenumerate
\etheorem

Of the five open questions regarding the convergence analysis for gradient sampling raised in \cite{BurkLewiOver05}, three were answered explicitly by Kiwiel in \cite{Kiwi07}.  Another open question was: ``Under what conditions can one guarantee that the GS algorithm terminates finitely?" This was posed in the context of a fixed sampling radius and therefore asks how one might know whether outcome (b) or (c) occurs in Theorem \ref{th.eps_pos}, assuming $f$ is bounded below. This remains open, but Kiwiel's introduction of the termination tolerances in the GS algorithm statement led to Corollary 3.1 which guarantees that when the sampling radius is reduced dynamically and the tolerances are nonzero, Algorithm \ref{alg.gs} must terminate if $f$ is bounded below. The only other open question concerns extending the convergence analysis to the non-Lipschitz case.

Overall, Algorithm~\ref{alg.gs} has a very satisfactory convergence theory in the locally Lipschitz case.  Its main weakness is its per-iteration cost, most notably due to the need to compute $m \geq n+1$ gradients in every iteration and solve a corresponding~QP.  However, enhancements to the algorithm have been proposed that can vastly reduce this per-iteration cost while maintaining these guarantees.  We discuss these and other enhancements in \S\ref{sec.enhancements}.

\subsection{A Local Linear Convergence Result}\label{sec.local} 

Given the relationship between gradient sampling and a traditional steepest descent approach, one might ask if there are scenarios in which Algorithm~\ref{alg.gs} can attain a linear rate of local convergence.  The study in~\cite{HeloSantSimo17} answers this in the affirmative, at least in a certain probabilistic sense. If $(i)$ the set of sampled points is \emph{good} in a certain sense described in \cite{HeloSantSimo17}, $(ii)$ the objective function $f$ belongs to a class of functions defined by the maximum of a finite number of smooth functions (``finite-max" functions), and $(iii)$ the input parameters are set appropriately, then Algorithm~\ref{alg.gs} will produce a step yielding a reduction in $f$ that is significant.  This analysis involves $\Vcal\Ucal$-decomposition ideas~\cite{LemaOustSaga00,Lewi02,MiffSaga05}, where in particular it is shown that the reduction in $f$ is comparable to that achieved by a steepest descent method restricted to the smooth $\Ucal$-space of $f$.  This means that a linear rate of local convergence can be attained over any infinite subsequence of iterations in which the sets of sampled points are \emph{good}.

\subsection{The Non-Lipschitz Case}\label{sec.nonlip}

In the non-locally Lipschitz case, the Clarke subdifferential $\clarkesub f$ is defined in \cite[p.~573]{BurkLewiOver02a};
unlike in the Lipschitz case, this set may be unbounded, presenting obvious difficulties for approximating it through
random sampling of gradients. In fact, more than half of \cite{BurkLewiOver02a} is devoted to investigating this issue,
relying heavily on modern variational analysis as expounded in \cite{RockWets98}.
Some positive results were obtained, specifically in the case that $f$ is ``directionally Lipschitz" at $\bar x$, which
means that the ``horizon cone" \cite[p.~572]{BurkLewiOver02a} of $f$ at $\bar x$ is pointed, that is, it does not contain a line. 
For example, this excludes the function on $\R{}$ defined by $f(x)=|x|^{1/2}$ at $\bar x=0$, but
it applies to the case $f(x)=(\max(0,x))^{1/2}$ even at $\bar x=0$. The discussion of the directionally
Lipschitz case culminates with Corollary 6.1, which establishes that the Clarke subdifferential can indeed be approximated
by convex hulls of gradients. On the more negative side, Example 7.2 shows that this approximation can fail badly in
the general Lipschitz case.  Motivated by these results, Burke and Lin have recently extended the GS convergence theory 
to the directionally Lipschitz case \cite{BurkLin17,Lin09}. However, it would seem difficult to
extend these results to the more general non-Lipschitz case.

Suppose $f:\R{n\times n}\rightarrow \R{}$ is defined by
\[
               f(X) = \max\{\mathrm{Re~} \lambda: \det(\lambda I - X) = 0\},
\]
the spectral abscissa (maximum of the real parts of the eigenvalues) of $X$. Assume that $\bar X$ has the
property that its only eigenvalues whose real parts coincide with $f$ make up a zero eigenvalue with multiplicity $q$
associated with a single Jordan block (the generic case).
In this case the results in \cite{BurkOver01} tell us that the horizon cone of $f$ is pointed at $\bar X$ if and only if the
multiplicity $q\leq 2$; on the other hand $f$ is locally Lipschitz at $X$ if and only if $q=1$.  In the light of the previous
paragraph, one might expect much greater practical success in applying GS to minimize the spectral abscissa of
a parameterized matrix if the optimal multiplicities are limited to 1 or 2. However, this seems not to be the case. The results
reported in \cite{BurkLewiOver02b} for unconstrained spectral abscissa minimization, as well as results for
applying  the algorithm of \cite{CurtOver12} (see  \S\ref{subsec.constrained} below)
for constrained nonsmooth, nonconvex optimization to problems
with spectral radius objective and constraints, as reported in \cite[Sec.~4.2 \red{and Appendix~A.1}]{CurtMitcOver17}, do not show any marked deterioration
as the optimal multiplicities increase from 2 or 3, although certainly the problems are much more challenging for
larger multiplicities. We view understanding the rather remarkably good behavior of the GS algorithm on such examples
as a potentially rewarding, though certainly challenging, line of investigation.

\section{Enhancements}\label{sec.enhancements}

As explained above, the statement of Algorithm~\ref{alg.gs} differs in several ways from the algorithms stated in \cite{BurkLewiOver02b}, \cite{BurkLewiOver05}, and \cite{Kiwi07}.  Other variants of the strategy have also been proposed in recent years, in some cases to pose new solutions to theoretical obstacles (such as the need, in theory, to check for differentiability of $f$ at each new iterate), and in others to enhance the practical performance of the approach.  In this section, we discuss a few of these enhancements.

\subsection{Restricting the Line Search to within a Trust Region}

Since the gradient information about $f$ is obtained only within the ball $\Bmbb(x^k,\epsilon_k)$ for all $k \in \N{}$, and since one might expect that smaller steps should be made when the sampling radius is small, an argument can be made that the algorithm might benefit by restricting the line search to within the ball $\Bmbb(x^k,\epsilon_k)$ for all $k \in \N{}$.  In \cite[\S4.2]{Kiwi07}, such a variant is proposed where in place of $-g^k$ the search direction is set as $-\epsilon_k g^k/\|g^k\|_2$.  With minor corresponding changes to conditions \eqref{eq.armijo} and \eqref{eq.update}, all of the theoretical convergence guarantees of the algorithm are maintained.  Such a variant with the trust region radius defined as a positive multiple of the sampling radius $\epsilon_k$ for all $k \in \N{}$ would have similar properties.  This variant might perform well in practice, especially in situations when otherwise setting the search direction as $-g^k$ would lead to significant effort being spent in the line search.

\subsection{Avoiding the Differentiability Check}\label{subsec.avoidcheck}

The largest distraction from the fundamentally simple nature of Algorithm~\ref{alg.gs} is the procedure for choosing a perturbed subsequent iterate if $f$ is not differentiable at $x^k - t_kg^k$; see Steps~\ref{step.check}--\ref{step.update}.  This procedure is necessary for the algorithm to be well defined since, to ensure that $-g^k$ is a descent direction for all $k \in \N{}$, the algorithm relies on the existence of and ability to compute $-\nabla f(x^k)$ for all $k \in \N{}$.  One might hope that this procedure, while necessary for theoretical convergence guarantees, could be ignored in practice.  However, due to the deterministic nature of the line search, situations exist in which landing on a point of nondifferentiability of $f$ occurs with positive probability.

For example, consider the function $f : \R{2} \to \R{}$ given by
\bequationNN
  f(w,z) = \max\{ 0.5w^2 + 0.1z,w + 0.1z + 1, -w + 0.1z + 1, -0.05z - 50 \};
\eequationNN
see Figure~\ref{lucas_example}.  As shown by Helou, Santos, and Sim\~oes in \cite{HeloSantSimo16}, if Algorithm~\ref{alg.gs} is initialized at $x^0 = (w_0,z_0)$ chosen anywhere in the unit ball centered at $(10,10)$, then there is a positive probability that the function $f$ will not be differentiable at $x^0 - t_0g^0$.  This can be explained as follows. At any point in the unit ball centered at $(10,10)$, the function $f$ is continuously differentiable and $\nabla f(x^0) = [w_0;\  0.1]^T$. Moreover, there is a positive probability that the sampled points obtained at this first iteration will yield $g^0 = \nabla f(x^0)$. Therefore, given a reasonable value for the parameter $\beta$ that appears in~\eqref{eq.armijo} (e.g., $\beta = 10^{-4}$), the sufficient decrease of the function value is attained with $t_0 = 1$. This guarantees that the function $f$ will not be differentiable at the next iterate, since the first coordinate of $x^1 = (w_1,z_1)$ will be zero.

\begin{figure}[ht]
  \centering
  \includegraphics[height=2.5in]{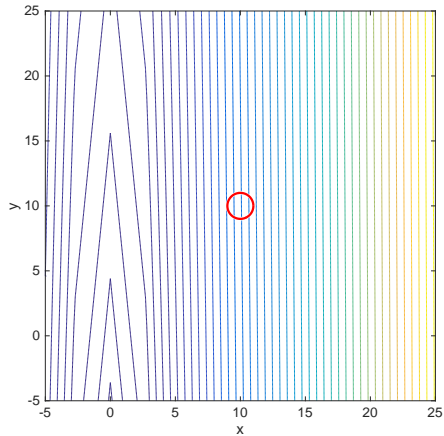}
  \caption{Contours of a function illustrating the necessity of the differentiability check in Algorithm~\ref{alg.gs}.  Initialized uniformly within the illustrated ball, there is a positive probability that $x^0 - t_0g^0 \not\in D$.}
  \label{lucas_example}
\end{figure}

The authors of \cite{HeloSantSimo16} propose two strategies to avoid the issues highlighted by this example.  The first is that, rather than perturb the iterate after the line search, one could perturb the search direction before the line search.  It is shown that if the random perturbation of the search direction is sufficiently small such that, for one thing, the resulting direction is still one of descent for $f$, then $f$ will be differentiable at all iterates with probability one.  Their second proposed strategy involves the use of a nonmonotone line search.  In particular, it is shown that if a strictly positive value $\Delta_k$ is added on the right-hand side of the sufficient decrease condition in \eqref{eq.armijo} such that $\{\Delta_k\}$ is summable, then one can remove $\nabla f(x^k)$ from the set $\Gcal_k$ for all $k \in \N{}$ and maintain convergence guarantees even when $f$ is \emph{not} differentiable at $\{x^k\}$.  This can be shown by noticing that, due to the positive term $\Delta_k$, the line search procedure continues to be well defined even if~$-g^k$ is not a descent direction for $f$ at $x^k$ (which may happen since $\nabla f(x_k)$ is no longer involved in the computation of~$g^k$). However, the summability of $\{\Delta_k\}$ implies that the possible increases in~$f$ will be finite and, when the sampled points are good enough to describe the local behavior of $f$ around the current iterate, the function value will necessarily decrease if the method has not reached (approximate) stationarity.  Overall, the sufficient reductions in $f$ achieved in certain iterations will ultimately exceed any increases.

Another proposal for avoiding the need to have $f$ differentiable at $x^k$ for all $k \in \N{}$ is given in \cite[\S4.3]{Kiwi07}, wherein a technique is proposed for using a limited line search, potentially causing the algorithm to take null steps in some iterations.  In fact, this idea of using a limited line search can also be used to avoid the need to sample a new set of $m \geq n+1$ gradients in each iteration, as we discuss next.

\subsection{Adaptive Sampling}\label{sec.adaptive}

As previously mentioned, the main weakness of Algorithm~\ref{alg.gs} is the cost of computing $m \geq n+1$ gradients in every iteration and solving a corresponding QP.  Loosely speaking, this lower bound on the number of gradients is required in the analysis of the algorithm so that one can use Carath\'eodory's theorem to argue that, with at least $n+1$ gradients, there is a sufficiently good chance that the combination vector $-g^k$ represents a sufficiently good direction from $x^k$.  However, in many situations in practice, the sampling of such a large number of gradients in each iteration can lead to a significant amount of wasted computational effort.  One can instead sample \emph{adaptively}, attempting to search along directions computed using fewer gradients and proceeding as long as a sufficient reduction is attained.

In \cite{CurtQue13}, Curtis and Que show how such an adaptive sampling strategy can be employed so that the convergence guarantees of Algorithm~\ref{alg.gs} are maintained while only a constant number (independent of $n$) of gradients need to be sampled in each iteration.  A key aspect that allows one to maintain these guarantees is the employment of a limited line search, as first proposed in \cite{Kiwi07}, potentially leading to a null step when fewer than $n+1$ gradients are currently in hand and
when the line search is not successful after a prescribed finite number of function evaluations.  See also \cite{CurtQue15} for further development of these ideas, where it is shown that one might not need to sample \emph{any} gradients as long as a sufficient reduction is attained.

The work in \cite{CurtQue13} also introduces the idea that, when adaptive sampling is employed, the algorithm can exploit a practical feature commonly used in bundle methods.  This idea relates to warm-starting the algorithm for solving the QP subproblems.  In particular, suppose that one has solved the primal-dual pair of QPs in \eqref{prob.qp} for some $m \in \N{}$ to obtain $(z_k,d^k,\lambda^k)$, yielding $g^k = -d^k$.  If one were to subsequently aim to solve the pair of QPs corresponding to the augmented matrix of gradients
\bequationNN
  \overline{G}_k = \bbmatrix G_k & \nabla f(x^{k,m+1}) & \cdots & \nabla f(x^{k,m+p}) \ebmatrix,
\eequationNN
then one obtains a viable feasible starting point for the latter QP in \eqref{prob.qp} by augmenting the vector $\lambda^k$ with~$p$ zeros.  This can be exploited, e.g., in an active-set method for solving this QP; see \cite{Kiwi86}.

As a further practical enhancement, the work in \cite{CurtQue13,CurtQue15} also proposes the natural idea that, after moving to a new iterate $x^k$, gradients computed in previous iterations can be ``re-used'' if they correspond to points that lie within $\Bmbb(x^k,\epsilon_k)$.  This may further reduce the number of sampled points needed in practice.

\subsection{Second-Order-Type Variants}

The solution vector $d^k$ of the QP in \eqref{prob.qp} can be viewed as the minimizer of the model of $f$ at $x^k$ given by
\bequationNN
  q_k(d) = f(x^k) + \max_{g\in\Gcal_k}\ g^Td + \thalf d^TH_kd
\eequationNN
with $H_k=I$, the identity matrix.  As in other second-order-type methods for nonlinear optimization, one might also consider algorithm variants where $H_k$ is set to some other symmetric positive definite matrix.  Ideas of this type have been explored in the literature.  For example, in \cite{CurtQue13}, two techniques are proposed: one in which $H_k$ is set using a quasi-Newton updating strategy and one in which the matrix is set in an attempt to ensure that the model $q_k$ represents an upper bounding model for $f$.  The idea of employing a quasi-Newton approach, inspired by the success of quasi-Newton methods in practice for nonsmooth optimization (see \cite{LewiOver13}), has also been explored further in \cite{CurtQue15,CurtRobiZhou17}.

Another approach, motivated by the encouraging results obtained when employing spectral gradient methods to solve smooth~\cite{Flet05,BirgMartRayd14} and nonsmooth~\cite{CremLoreRayd07} optimization problems, has been to employ a Barzilai-Borwein (BB) strategy for computing initial stepsizes in a GS approach; see \cite{LoreAponCoreRayd17} and the background in~\cite{BarzBorw88,Rayd93,Rayd97}.  Using a BB strategy can be viewed as choosing $H_k = \alpha_k I$ for all $k \in \N{}$ where the scalar $\alpha_k$ is set according to iterate and gradient displacements in the latest iteration.

In all of these second-order-type approaches, one is able to maintain convergence guarantees of the algorithm as long as the procedure for setting the matrix $H_k$ is safeguarded during the optimization process.  For example, one way to maintain guarantees is to restrict each $H_k$ to the set of symmetric matrices whose eigenvalues are contained within a fixed positive interval.  One might also attempt to exploit the self-correcting properties of BFGS updating; see \cite{CurtRobiZhou17}.

\section{Extensions}\label{sec.extensions}

In this section, we discuss extensions to the GS methodology to solve classes of problems beyond unconstrained optimization on $\R{n}$.

\subsection{Riemannian GS for Optimization on Manifolds}

Hosseini and Uschmajew in \cite{HossUsch17} have extended the GS methodology for minimizing a locally Lipschitz $f$ over a set $\Mcal$, where $\Mcal$ is a complete Riemannian manifold of dimension $n$.  The main idea of this extension is to employ the convex hull of gradients from tangent spaces at randomly sampled points \emph{transported} to the tangent space of the current iterate.  In this manner, the algorithm can be characterized as a generalization of the Riemannian steepest descent method just as GS is a generalization of traditional steepest descent.  Assuming that the vector transport satisfies certain assumptions, including a \emph{locking condition}, the algorithm attains convergence guarantees on par with those for Algorithm~\ref{alg.gs}.

\subsection{SQP-GS for Constrained Optimization}
\label{subsec.constrained}

Curtis and Overton in \cite{CurtOver12} proposed a combination sequential quadratic programming (SQP) and gradient sampling method for solving constrained optimization problems involving potentially nonsmooth constraint functions, i.e.,
\bequation\label{eq.huatulco}
  \min_{x\in\R{n}}\ f(x)\ \ \st\ \ c(x) \leq 0,
\eequation
where $f : \R{n} \to \R{}$ and $c : \R{n} \to \R{n_c}$ are locally Lipschitz.   A key aspect of the proposed SQP-GS approach is that sampled points for the objective and each individual constraint function are generated independently.  With this important feature, it is shown that the algorithm, which follows a penalty-SQP strategy (e.g., see~\cite{Flet87}), attains convergence guarantees for minimizing an exact penalty function that are similar to those in \S\ref{sec.global}.  Moreover,  with the algorithm's penalty parameter updating strategy, it is shown that either the penalty function is driven to $-\infty$, the penalty parameter settles at a finite value and any limit point will be feasible for the constraints and stationary for the penalty function, or the penalty parameter will be driven to zero and any limit point of the algorithm will be stationary for a constraint violation measure.  As for other exact penalty function methods for nonlinear optimization, one can translate between guarantees for minimizing the exact penalty function and solving the constrained problem~\eqref{eq.huatulco}; in particular, if the problem is \emph{calm} \cite{Clar83,Rock82}, then at any local minimizer~$x_*$ of~\eqref{eq.huatulco} there exists a threshold for the penalty parameter beyond which $x_*$ will be a local minimizer of the penalty function.

Tang, Liu, Jian, and Li have also proposed in \cite{TangLiuJianLi14} a \emph{feasible} variant of the SQP-GS method in which the iterates are forced to remain feasible for the constraints and the objective function is monotonically decreasing throughout the optimization process.  This opens the door to employing a two-phase approach common for solving some optimization problems, where phase 1 is responsible for attaining a feasible point and phase 2 seeks optimality while maintaining feasibility.

\subsection{Derivative-Free Optimization}

Given its simple nature, gradient sampling has proved to be an attractive basis for the design of new algorithms even when gradient information cannot be computed explicitly.  Indeed, there have been a few variants of \emph{derivative-free} algorithms that have been inspired by gradient sampling.

The first algorithm for derivative-free optimization inspired by gradient sampling was proposed by Kiwiel in~\cite{Kiwi10}.  In short, in place of the gradients appearing in Algorithm~\ref{alg.gs}, this approach employs Gupal's estimates of gradients of the Steklov averages of $f$.  In this manner, function values only---specifically, $\Ocal(mn)$ per iteration---are required for convergence guarantees.  A less expensive \emph{incremental} version is also proposed.

Another derivative-free variant of GS, proposed by Hare and Nutini in \cite{HareNuti13}, is specifically designed for minimizing finite-max functions.  This approach exploits knowledge about which of these functions are \emph{almost active}---in terms of having value close to the objective function---at a particular point.  In so doing, rather than attempt to approximate gradients at nearby points, as in done in \cite{Kiwi10}, this approach only attempts to approximate gradients of almost active functions.  The convergence guarantees proved for the algorithm are similar to those for GS methods, though the practical performance is improved by the algorithm's tailored gradient approximation strategy.

Finally, we mention the \emph{manifold sampling} algorithm, proposed by Larson, Menickelly, and Wild in \cite{LarsMeniWild16}, for solving nonconvex problems where the objective function is the $\ell_1$-norm of a smooth vector function $F : \R{n} \to \R{r}$. While this approach does not employ a straightforward GS methodology in that it does not randomly sample points, it does employ a GS-type approach in the way that the gradient of a model of the objective function is constructed by solving a QP of the type in \eqref{prob.qp}.  Random sampling can be avoided in this construction since the algorithm can exploit knowledge of the signs of the elements of $F(x)$ at any $x \in \R{n}$ along with knowledge of $\clarkesub \|\cdot\|_1$.

\section{Applications}\label{sec.applications}

We mentioned in the introduction that the original gradient sampling paper \cite{BurkLewiOver02b} reported results for spectral abscissa optimization problems that had not been solved previously. The second gradient sampling paper \cite{BurkLewiOver05} reported
results for many more applications that again had not been solved previously: these included Chebyshev approximation by exponential
sums, eigenvalue product minimization for symmetric matrices, spectral and pseudospectral abscissa minimization, maximization of
the ``distance to instability", and fixed-order controller design by static output feedback. 

Subsequently, the GS algorithm played a key role in the \hanso\ (Hybrid Algorithm for Non-Smooth Optimization)\footnote{www.cs.nyu.edu/overton/software/hanso/}
and \hifoo\ (H-Infinity Fixed-Order Optimization)\footnote{www.cs.nyu.edu/overton/software/hifoo/} toolboxes. The former is a stand-alone code for unconstrained nonsmooth optimization while the latter is a more
specialized code used for the design of low-order controllers for linear dynamical systems with input and output,
computing fixed-order controllers by optimizing stability measures that are generally nonsmooth at local minimizers \cite{BurkHenrLewiOver06}.
\Hifoo\ calls \hanso\ to carry out the optimization.
The use of ``hybrid" in the expansion of the \hanso\ acronym indicated that, from its inception, \hanso\ combined
the use of both a quasi-Newton algorithm (BFGS) and Gradient Sampling.
The quasi-Newton method was used in an initial phase
which, rather surprisingly, typically worked very
effectively even in the presence of nonsmoothness,
very often providing a fast way to
approximate a local minimizer. This was followed by
a GS phase to refine the approximation, typically verifying a loose measure of local optimality.
The \hifoo\ toolbox has been used successfully in a wide variety of applications, including
synchronization of heterogeneous multi-agent systems and networks,
design of motorized gimbals that stabilize an angular motion of an optical
payload around an axis, 
flight control via static output feedback,
robust observer-based fault detection and isolation,
influence of tire damping on control of quarter-car suspensions,
flexible aircraft lateral flight dynamic control,
optimal control of aircraft with a blended wing body,
vibration control of a fluid/plate system,
controller design of a nose landing gear steering system, 
bilateral teleoperation for minimally invasive surgery,
design of an aircraft controller for improved gust
alleviation and passenger comfort,
robust controller design for a proton exchange membrane fuel cell
system, design of power systems controllers, and design of winding systems for elastic web
materials --- for a full list of references, see \cite{CurtMitcOver17}. 

The successful use of BFGS in \hanso\ and \hifoo\ led to papers on the use of quasi-Newton methods in the
nonsmooth context, both for unconstrained \cite{LewiOver13} and constrained \cite{CurtMitcOver17} optimization.
The latter paper introduced a new BFGS-SQP method for nonsmooth constrained optimization and compared it with the SQP-GS method discussed
in \S~\ref{subsec.constrained} on a suite of challenging static output feedback controller design problems, half of them
non-Lipschitz (spectral radius minimization) and half of them locally Lipschitz (pseudospectral radius minimization).
It was found that although the BFGS-SQP method was much faster than SQP-GS, nonetheless, if the latter method based on Gradient Sampling
was allowed sufficient running time, it frequently found better approximate solutions than the former method based on BFGS
in a well defined sense, evaluated using ``relative minimization profiles".
Interestingly, this was particularly pronounced on the non-Lipschitz problems, despite the fact that the GS convergence theory
does not extend to this domain. See \S\ref{sec.nonlip} for further discussion of this issue.

Finally, we mention an interesting application of GS to robot path planning \cite{TrafMitc16}.
This work is based on the observation that 
shortest paths generated through gradient descent
on a value function have a tendency to chatter and/or require
an unreasonable number of steps to synthesize. The authors demonstrate that the GS algorithm
can largely alleviate this problem. For systems
subject to state uncertainty whose state estimate is tracked
using a particle filter, they proposed the Gradient Sampling with
Particle Filter (GSPF) algorithm, which uses the particles as
the locations in which to sample the gradient. At each step,
the GSPF efficiently finds a consensus direction suitable for all
particles or identifies the type of stationary point on which it
is stuck. If the stationary point is a minimum, the system has
reached its goal (to within the limits of the state uncertainty)
and the algorithm terminates; otherwise, the authors propose
two approaches to find a suitable descent direction. They illustrated
the effectiveness of the GSPF on several examples using well known
robot simulation environments. This work was recently extended and modified in \cite{EstrMitc18}, where
the practical effectiveness of both the GSPF algorithm and the new modification was demonstrated on a Segway Robotic Mobility Platform.

\section{Conclusion and Future Directions}\label{sec.conclusion}

Gradient sampling is a conceptually straightforward approximate steepest descent method.
With a solid convergence theory, the method has blossomed into a powerful methodology for solving nonsmooth minimization problems. 
The theme of our treatment of GS in this work has been to emphasize the fact that, even though the basic
algorithm has been enhanced and extended in various ways, the foundation of the approach is fundamentally simple in nature.

We have also corrected an oversight in the original GS theory (i.e., that the convergence results depend on assuming that the set of points over which the Lipschitz function $f$ is continuously differentiable has full measure, although we do not have a counterexample to convergence of GS in the absence of this assumption).
At the same time we have loosened the requirements of the algorithm (showing that $f$ need only be differentiable at the iterates and sampled points).  An open question that still remains is whether one can extend the GS theory to broader function classes, such as the case where $f$ is assumed to be semi-algebraic but not necessarily locally Lipschitz or directionally Lipschitz.

Opportunities for extending GS theory for broader function classes may include connecting the algorithm to other randomized/stochastic optimization methods.  For example, one might 
view the algorithm as a stochastic-gradient-like method applied to a smoothed objective.  (A similar philosophy underlies the analysis by 
Nesterov and Spokoiny in \cite{NestSpok17}).)  More precisely, given a locally Lipschitz objective  $f$,  consider a smoothing~$f_\epsilon$  whose value at any point  $x$  is given by the mean value of  $f$  over the ball $\Bmbb(x,\epsilon)$.  The GS algorithm uses gradients of  $f$  at uniformly distributed random points in this ball.  
Notice that each such gradient can also be viewed as a stochastic gradient for the smoothing $f_\epsilon$ in the sense that its expectation is the gradient of $f_\epsilon$ at $x$.  Thus, one might hope to prove convergence results for a GS algorithm (with predetermined stepsizes rather than line searches) that parallel convergence theory for stochastic gradient methods. Recent work by Davis, Drusvyatskiy, Kakade and Lee \cite{DaviDrusKakaLee18} gives convergence results for stochastic \emph{subgradient} methods on
a broad class of problems.

Another potentially interesting connection is with the work of Davis and Drusvyatskiy \cite{DaviDrus18} on stochastic model-based optimization. 
Consider a GS variant that successively minimizes stochastic models of the objective function $f$, where we assume for simplicity that $f$ is a globally Lipschitz convex function.
In this variant, rather than moving along the direction $-g^k$, consider instead the construction of a cutting plane approximation of $f$ from its affine minorants at the current iterate $x^k$ and the sampled points $\{x^{k,i}\}$, augmented by the proximal term $\beta_k \|x - x^k\|^2$, where $\{\beta_k\}$ is a predetermined sequence.
Suppose that the next iterate is chosen as the minimizer of this model; for a given $k$ and with $\beta_k = 1$, by equation \eqref{prob.qp}, 
this scheme and GS produce similar descent directions as the sampling radius tends to zero.
It follows from the results of \cite{DaviDrus18} that the expected norm of the gradient of the Moreau envelope \cite[p.~6]{DaviDrus18} is reduced below $\epsilon$ in 
$\Ocal(\epsilon^{-4})$ iterations.
In fact, the assumptions on $f$ in  \cite{DaviDrus18} are substantially weaker than convexity, and do not require any property of the set on which $f$ is continuously differentiable.

Connecting the convergence theory for GS to stochastic methods as suggested in the previous two paragraphs could be enlightening.  However, while stochastic methods are often designed for settings in which it is intractable to compute function values exactly---a feature reflected in the fact that the analyses for such methods are based on using predetermined stepsize sequences---the GS methodology has so far been motivated by problems for which functions and gradients are tractable to compute.  In such settings, the line search in Algorithm~GS is an ingredient that is crucial to its practical success.

\appendix
\section{On the Assumption of Full Measure of the Set $D$}\label{app.fullmeasure}

Recall that $D$ is defined to be the set of points on which the locally Lipschitz function $f$ is continuously differentiable.
There are two ways in which the analyses in \cite{Kiwi07,BurkLewiOver05} actually depend on $D$ having full measure:
\begin{enumerate}
\item The most obvious is that both papers require that the points sampled in each iteration should lie in~$D$, and a statement is made in both papers that this occurs with probability one, but this is not the case if $D$ is assumed only to be an open dense subset of $\R{n}$. However, as already noted earlier and justified in Appendix \ref{app.fdifOK}, this requirement can be relaxed, as in Algorithm~\ref{alg.gs} given in \S\ref{sec.algorithm}, to require only that $f$ be differentiable at the sampled points.

\item The set $D$ must have full measure for Property \ref{prop}, stated below, to hold.  The proofs in \cite{BurkLewiOver05,Kiwi07} depend critically on this property, which follows from \cite[Eq.~1.2]{BurkLewiOver02a} (where it was stated without proof).
For completeness we give a proof here, followed by an example that
demonstrates the necessity of the full measure assumption. 
\end{enumerate}

\begin{prop}\label{prop}
Assume that $D$ has full measure and let
  \[
       G_\epsilon(x):= \cl\co\nabla f\left( \Bmbb(x,\epsilon)\cap D \right).
\]
For all $\epsilon > 0$ and all $x\in \R{n}$, one has $\clarkesub f(x) \subseteq G_{\epsilon}(x)$, where the Clarke subdifferential 
$\clarkesub f$ is defined in \eqref{eq.clarkesub}.
\end{prop}

\begin{proof}
  Let $x\in \R{n}$ and $v\in \clarkesub f(x)$.  We have from  \cite[Thm 2.5.1]{Clar83} that, for any set $S$ with zero measure,
\[
\clarkesub f(x) = \co\left\{ \lim_j \nabla f(y^j) : y^j\to x\ \text{where}\ y^j\notin S\ \text{and}\ f\ \text{is differentiable at}\  y^j, \text{for all}\ j \in \N{}\right\}.
\]
In particular, since $D$ has full measure and $f$ is differentiable on $D$, it follows that
\[
\clarkesub f(x) = \co\left\{ \lim_j \nabla f(y^j) : y^j\to x\ \text{with}\ y^j\in D\ \text{for all}\ j \in \N{}\right\}\text.
\]
Considering this last relation and Carath\'{e}odory's theorem, it follows that $v \in \co\left\{ \xi^1,\hdots,\xi^{n+1} \right\}$, where, for all $i\in \{1,\hdots,n+1\}$, one has $\xi^i = \lim_j \nabla f(y^{j,i})$ for some sequence $\{y^{j,i}\}_{j\in \N{}} \subset D$ converging to $x$. Hence, there must exist a sufficiently large $j_i \in \N{}$ such that
\[
y^{j,i}\in \Bmbb(x,\epsilon)\cap D \implies \nabla f(y^{j,i})\in \nabla f\left(\Bmbb(x, \epsilon)\cap D\right) \subseteq \co\nabla f\left(\Bmbb(x, \epsilon)\cap D\right)\ \ \text{for all}\ \ j\geq j_i.
\]
Recalling that $G_\epsilon (x)$ is the closure of $\co\nabla f\left(\Bmbb(x, \epsilon)\cap D\right)$, it follows that $\xi^i\in G_\epsilon (x)$ for all $i\in \{1,\hdots,n+1\}$.  Moreover, since $G_\epsilon (x)$ is convex, it follows that $v\in G_\epsilon (x)$.  The result follows since $x \in \R{n}$ and $v\in \clarkesub f(x)$ were arbitrarily chosen.
\end{proof}

With the assumption that $D$ has full measure, Property \ref{prop} holds and hence the proofs of the results in \cite{BurkLewiOver05,Kiwi07} are all valid. In particular, the proof of (ii) in \cite[Lemma~3.2]{Kiwi07}, which borrows from \cite[Lemma~3.2]{BurkLewiOver05}, depends on Property \ref{prop}. See also the top of \cite[p.~762]{BurkLewiOver05}.


The following example shows that Property~\ref{prop} might not hold if $D$ is assumed only to be an open dense set, not necessarily of full measure.
\begin{examp}
Let $\delta \in (0,1)$ and $\{q_k\}_{k\in \mathbb{N}}$ be the enumeration of the rational numbers in $(0,1)$. Define
\[
D:= \bigcup_{k=1}^\infty \mathcal{Q}_k\text{, where }\mathcal{Q}_k:=\left(q_k - \frac{\delta}{2^{k+1}} , q_k + \frac{\delta}{2^{k+1}}\right)\text.
\]
Clearly, its Lebesgue measure satisfies $0<\lambda(D) \leq \delta$. Moreover, the set $D$ is an open dense subset of $[0,1]$. Now, let $i_{D}:[0,1]\to \mathbb{R}$ be the indicator function of the set $D$,
\[
i_{D}(x) = \left\{
\begin{array}{cc}
1\text, & \text{if}\ x\in D \\
0\text, & \text{if}\ x\notin D
\end{array}\text.
\right.
\]
Then, considering the Lebesgue integral, we define the function $f:[0,1]\to \mathbb{R}$,
\[
f(x) = \int_{[0,x]} i_{D}d\lambda\text.
\]
Let us prove that $f$ is a Lipschitz continuous function on $(0,1)$. To see this, note that given any $a,b\in (0,1)$ with $b>a$, it follows that
\[
|f(b) - f(a)| = \left| \int_{[0,b]} i_{D}d\lambda - \int_{[0,a]} i_{D}d\lambda \right|
              = \left| \int_{(a,b]} i_{D}d\lambda \right|
              \leq \int_{(a,b]} 1d\lambda
              = b-a\text,
\]
which ensures that $f$ is a Lipschitz continuous function on $(0,1)$. Consequently, the Clarke subdifferential set of $f$ at any point in $(0,1)$ is well defined. Moreover, we claim that,
 for all $k\in \mathbb{N}$, $f$ is continuously differentiable at any point $q\in \mathcal{Q}_k$ and the following holds 
\begin{equation}\label{eq:example1}
\tag{1}
f'(q) = i_{D}(q) = 1\text.
\end{equation}
Indeed, given any $q\in \mathcal{Q}_k$, we have
\[
f(q+t) - f(q) = \int_{[0,q+t]}i_{D}d\lambda - \int_{[0,q]}i_{D}d\lambda = \int_{(q,q+t]}i_{D}d\lambda\text{, for }t>0\text. 
\]
Since $\mathcal{Q}_k$ is an open set, we can find $\overline t>0$ such that $[q,q+t]\subset \mathcal{Q}_k\subset D$, for all $t\leq \overline t$. Hence, given any $t\in (0,\overline {t}]$, it follows that 
\[
f(q+t) - f(q) = \int_{(q,q+t]} 1d\lambda = t\ \Longrightarrow \lim_{t\searrow 0}\frac{f(q+t) - f(q)}{t} = 1 = i_{D}(q)\text.
\]
The same reasoning can be used to see that the left derivative of $f$ at $q$ exists and it is equal to $i_{D}(q)$. Consequently, we have $f'(q) = i_{D}(q) = 1$ for all $q\in \mathcal{Q}_k$, which yields that $f$ is continuously differentiable on $D$.

By the Lebesgue differentiation theorem, we know that $f'(x) = i_{D}(x)$ almost everywhere. Since the set $[0,1]\setminus D$ does not have measure zero, this means that there must exist $z\in [0,1]\setminus D$ such that $f'(z) = i_{D}(z) = 0$. Defining $\epsilon := \min\{z,1-z\}/2$, we see, by~\eqref{eq:example1}, that the set
\[
G_\epsilon(z):= \cl\conv\nabla f([z-\epsilon,z+\epsilon]\cap D)
\]
is a singleton $G_\epsilon(z)=\{1\}$. However, since $f'(z) = 0$, it follows that $0\in \clarkesub f(z)$, which implies $\clarkesub f(z)\not\subset G_\epsilon(z)$.  
\end{examp}


Note that it is stated on \cite[p.~754]{BurkLewiOver05} and \cite[p.~381]{Kiwi07} that the following holds: 
for all $0 \leq \epsilon_1 < \epsilon_2$ and all $x\in\R{n}$, one has $\bar\partial_{\epsilon_1} f(x) \subseteq G_{\epsilon_2}(x)$.
Property \ref{prop} is a special case of this statement with $\epsilon_1=0$, and hence this statement too holds only under the
full measure assumption. 

Finally, it is worth mentioning that in practice, the full measure assumption on $D$ usually holds. In particular, whenever a real-valued function is semi-algebraic (or, more generally, ``tame") --- in other words, for all practical purposes virtually always --- it is continuously differentiable on an open set of full measure.  Hence, the original proofs hold in such contexts.

\section{On Relaxing the Requirement that the Objective is Continuously Differentiable at the Iterates and Sampled Points}\label{app.fdifOK}

In this appendix, we summarize why it is not necessary that the iterates and sampled points of the algorithm lie in the set $D$ in which $f$ is continuously differentiable, and that rather it is sufficient to ensure that $f$ is differentiable at these points, as in Algorithm \ref{alg.gs}. We do this by outlining how to modify the proofs in \cite{Kiwi07} to extend to this case.
\begin{enumerate}
\item That the gradients at the sampled points $\{x^{k,j}\}$ exist follows with probability one from Rademacher's theorem, while the existence of the gradients at the iterates $\{x^k\}$ is ensured by the statement of Algorithm \ref{alg.gs}.  Notice that the proof of part (ii) of \cite[Theorem~3.3]{Kiwi07} still holds in our setting with the statement that the components of the sampled points are ``sampled independently and uniformly from
$\Bmbb(x^k,\epsilon)\cap D$" replaced with ``sampled independently and uniformly from $\Bmbb(x^k,\epsilon)$''.

\item One needs to verify that $f$ being differentiable at $x^k$ is enough to ensure that the line search procedure presented in~\eqref{eq.armijo} terminates finitely. This is straightforward. Since $\nabla f(x^k)$ exists, it follows that the directional derivative along any vector $d\in \R{n}\setminus\{0\}$ is given by $f'(x^k; d) = \nabla f(x^k)^Td$. Furthermore, since $-\nabla f(x^k)^Tg^k \leq -\|g^k\|_2^2$ (see \cite[p.\ 756]{BurkLewiOver05}), it follows, for any $\beta \in (0,1)$, that there exists $\overline t>0$ such that
  \bequationNN
    f(x^k-tg^k) < f(x^k) - t\beta\|g^k\|_2^2\ \ \text{for any}\ \ t \in (0,\overline t).
  \eequationNN
  This shows that the line search is well defined.
\item The only place where we actually need to modify the proof in \cite{Kiwi07} concerns item (ii) in Lemma 3.2, where it is stated that
$\nabla f(x^k) \in G_\epsilon(\overline x)$ (for a particular point $\overline x$) because $x^k \in \Bmbb(\overline x,\epsilon/3) \cap D$; the latter is not true if $x^k \not \in D$.  However, using Property \ref{prop}, we have 
\bequationNN
  \nabla f(x^k)\in \clarkesub f(x^k)\subset G_{\epsilon/3}(x^k) \subset G_{\epsilon}(\overline x) \text{ when } x^k\in \Bmbb(\overline x,\epsilon/3),
\eequationNN
so $\nabla f(x^k) \in G_\epsilon(\overline x)$ even when $x^k\not \in D$.
\end{enumerate}

Finally, although it was convenient in Appendix A to state Property~1 in terms of $D$, it actually holds if $D$ is replaced by any full measure set
on which $f$ is differentiable.
Nonetheless, it is important to note that the proofs of the results in \cite{BurkLewiOver05,Kiwi07} \emph{do} require that $f$ be \emph{continuously} differentiable on $D$. This assumption is used in the proof of (i) in \cite[Lemma 3.2]{Kiwi07}.

\bibliographystyle{alpha}
\bibliography{gradient_sampling}

\end{document}